\documentclass[a4paper]{amsart}

\usepackage{amsmath,amstext,amssymb,mathrsfs,amscd,amsthm,indentfirst, bbm}
\usepackage{amsfonts}

\usepackage{booktabs}
\usepackage{verbatim}
\usepackage{enumerate}

\usepackage{graphicx}
\newcommand{\bC}{\mathbb{C}}

\newcommand{\bR}{\mathbb{R}}

\newcommand{\bZ}{\mathbb{Z}}

\newcommand\lra{\longrightarrow}

\newcommand\Diff{\mathrm{Diff}}

\newcommand{\hcoker}{/\!\!/}

\newcommand{\R}{\bR}

\renewcommand{\epsilon}{\varepsilon}

\mathchardef\ordinarycolon\mathcode`\:
\mathcode`\:=\string"8000
\begingroup \catcode`\:=\active
  \gdef:{\mathrel{\mathop\ordinarycolon}}
\endgroup

\usepackage{tikz}
\usepackage{tikz-cd}

\usepackage{amsthm}
\theoremstyle{plain}

\newtheorem{theorem}{Theorem}[section]

\newtheorem{proposition}[theorem]{Proposition}
\newtheorem{lemma}[theorem]{Lemma}

\newtheorem{corollary}[theorem]{Corollary}
\newtheorem{question}[theorem]{Question}

\theoremstyle{definition}

\theoremstyle{remark}

\newtheorem*{remark*}{Remark}

\numberwithin{equation}{section}

\title{On the order of Dehn twists}

\author{Ailsa Keating}
\email{amk50@cam.ac.uk}

\author{Oscar Randal-Williams}
\thanks{ORW was  partially  supported  by  the  ERC  under the European Union’s Horizon 2020 research and innovation programme (grant agreement No. 756444), and by a Philip Leverhulme Prize from the Leverhulme Trust.}
\email{o.randal-williams@dpmms.cam.ac.uk}

\address{Centre for Mathematical Sciences\\
Wilberforce Road\\
Cambridge CB3 0WB\\
UK}
\date{\today}

\begin{document}
\begin{abstract}
This note records the order of a higher dimensional Dehn twist in a range of topologically significant groups. 
\end{abstract}
\maketitle

\section{Introduction}

Consider $$T^*S^n = \{ (u,v) \in \bR^{n+1} \times \bR^{n+1} \, | \,  |u|=1, \langle u,v \rangle =0 \}$$
equipped with its standard exact symplectic form.
Recall that the \emph{Dehn twist} in the zero section is defined as follows  \cite{Arnold}.   Let $\mu(u,v) = ||v||$. (After identifying $T^\ast S^n$ with $TS^n$ via the standard metric, this is the Hamiltonian generating the normalised geodesic flow on $(T S^n )\backslash S^n$.) Let $\rho \in C^\infty (\bR)$ be a function which vanishes for sufficiently large $t$ and such that 
$\rho(-t) = \rho(t)-t$ for small $|t|$, and let $( \phi_t)$ be the Hamiltonian flow of $\rho(\mu)$. Then $\phi_{2\pi}$ extends smoothly over the zero section to a compactly-supported symplectomorphism of $T^\ast S^n$. This is the Dehn twist $\tau$. It represents an element $[\tau] \in \pi_0(\mathrm{Symp}_c(T^* S^n))$ in the compactly-supported symplectic mapping class group which is independent of the choice of $\rho$. 

 Under the neglectful homomorphisms
$$\pi_0(\mathrm{Symp}_c(T^* S^n)) \to \pi_0(\mathrm{Diff}_c(T^* S^n)) \to \pi_0(\mathrm{Homeo}_c(T^* S^n)) \to \pi_0(\mathrm{hAut}_c(T^* S^n))$$
there are elements also referred to as Dehn twists in the compactly-supported smooth, topological, and homotopical mapping class groups. The goal of this note is to record the orders of these elements.

\begin{theorem}\label{thm:main}
If $n$ is odd then $[\tau]$ has infinite order in all of these mapping class groups.

If $n$ is even then
\begin{enumerate}[(i)]
\item $[\tau] \in \pi_0(\mathrm{Symp}_c(T^* S^n))$ has infinite order.

\item  $[\tau] \in \pi_0(\mathrm{Diff}_c(T^* S^n))$ has order
\begin{enumerate}[(a)]
\item 2 if $n=2$ or $6$, and otherwise,

\item 4 if the $(2n+1)$-dimensional Kervaire sphere is trivial,

\item 8 if the $(2n+1)$-dimensional Kervaire sphere is not trivial.

\end{enumerate}

\item $[\tau] \in \pi_0(\mathrm{Homeo}_c(T^* S^n))$ has order
\begin{enumerate}[(a)]
\item 2 if $n=2$ or $6$,

\item 4 otherwise. 
\end{enumerate}

\item $[\tau] \in \pi_0(\mathrm{hAut}_c(T^* S^n))$ has order
\begin{enumerate}[(a)]
\item 2 if $n=2$ or $6$,

\item 4 otherwise.
\end{enumerate}
\end{enumerate}
\end{theorem}

Recall that the Kervaire spheres of dimensions 5, 13, 29, and 61 are trivial, that the one of dimension 125 may or may not be trivial, and that all others are not trivial \cite{HHR}.

In Section \ref{sec:ACDiffeo} we shall also discuss the order of $\tau$ in the ``almost-complex mapping class group". We do not manage to pin down the order completely, but we show that it is finite and is approximately $n!$.

We do not claim any originality for most of these results: (i) is due to Seidel \cite[Theorem 4.17]{Seidel_graded}, as is the $n=2$ case of (ii) \cite[Lemma 6.3]{Seidel}, where it is attributed to Kronheimer; the rest of (ii) is due to Kauffman--Krylov \cite{KK}. Our contribution, and goal, is to popularise the results of Kauffman--Krylov among sympletic topologists, where they do not seem to be well know, and to explain how these methods imply (iii) and (iv).

\section{The action on a cotangent fibre}\label{seq:CotangentAction}

Let us consider $T^*S^n$ as the interior of a closed disc bundle $D^*(S^n)$, with boundary $S^*(S^n)$, and let $B = S^{n-1} \subset S^*(S^n)$ be the fibre over a basepoint $* \in S^n$. Then there is a $\pi_0(\mathrm{hAut}_c(T^* S^n))$-equivariant short exact sequence
$$0 \to \bZ\{[S^n]\}=H_n(D^*(S^n);\bZ) \to H_n(D^*(S^n), B;\bZ) \to H_{n-1}(B;\bZ) =\bZ\{[B]\}\to 0.$$
This is split by $s([B]) = [D]$ for $D \subset D^*(S^n)$ the fibre over $* \in S^n$. In terms of the basis $\{[S^n], [D]\}$ the element $\tau$ acts on $H_n(D^*(S^n), B;\bZ)$ as
$$\begin{pmatrix}
(-1)^{n+1} & (-1)^{\tfrac{(n+1)(n+2)}{2}} \\
0 & 1
\end{pmatrix},$$
(see \cite[Ch.\ 2, \S 1.3]{ArnoldBook}, and use the choices of orientations made there) so we see that $[\tau] \in \pi_0(\mathrm{hAut}_c(T^* S^n))$ has infinite order if $n$ is odd, and if $n$ is even then 2 divides its order.

If $n$ is odd then of course $[\tau]$ must also have infinite order in the topological, smooth, and symplectic mapping class groups, which proves the odd $n$ part of Theorem \ref{thm:main}. It also provides 2 as a lower bound for the order of $\tau$ when $n$ is even. For later use we record the following further consequence.

\begin{lemma}\label{lem:CotangentAction}
When $n$ is even, each element of $\pi_0(\mathrm{hAut}_c(T^* S^n))$ acts as either
$$\begin{pmatrix}
-1 & (-1)^{\tfrac{n}{2}+1} \\
0 & 1
\end{pmatrix}$$
or as the identity on $H_n(D^*(S^n), B;\bZ)$, in terms of the basis $\{[S^n], [D]\}$.
\end{lemma}
\begin{proof}
Any $[f] \in \pi_0(\mathrm{hAut}_c(T^* S^n))$ acts on $[S^n]$ by $\epsilon \in \{+1, -1\}$, and fixes the class $[B]$ and so satisfies $f_*[D] = A [S^n] + [D]$ for some $A \in \bZ$. The intersection pairing
$$\langle -, - \rangle : H_n(D^*(S^n), B;\bZ) \otimes H_n(D^*(S^n);\bZ) \to H_0(D^*(S^n);\bZ)=\bZ$$
satisfies $\langle [D], [S^n] \rangle=1$ and $\langle [S^n], [S^n] \rangle = 2 (-1)^{n/2}$ with the choice of orientations we are using \cite[p.\ 56]{ArnoldBook}, so we have
$$1 = \langle [D], [S^n] \rangle = \langle f_*[D], f_*[S^n] \rangle = \langle A[S^n] + [D], \epsilon[S^n] \rangle = (2(-1)^{n/2}A+1)\epsilon$$
and hence $A = \frac{\epsilon-1}{2} (-1)^{n/2}$ as required.
\end{proof}

\section{The cases $n=2$ and $6$}\label{sec:nIs2}

\begin{proposition}
Assume $n=2$ or $6$. Then $[\tau] \in \pi_0 (\Diff_c (T^\ast S^n))$ has order 2.
\end{proposition}

\begin{proof} For the $n=2$ case, this is   \cite[Lemma 6.3]{Seidel}, attributed to Kronheimer. The proof uses the cross-product structure on $\R^3$; we reproduce it here with a slight rephrasing to make it clear that it also works for the $n=6$ case. Let $\times$ denote the standard cross-product on $\R^3$, or a choice of cross-product on $\R^7$ (see \cite[\S 2]{Calabi}); we will make clear which properties we use.

Suppose $u, v \in \R^{n+1}$ are linearly independent. Let $R^\theta_{(u,v)}$ denote the transformation of $\R^{n+1}$ which fixes the orthogonal complement of the plane  $\langle u, v \rangle$, and rotates the plane $\langle u, v \rangle$ by angle $\theta$, with orientation chosen so that $(u,v)$ is positive. Now consider the round sphere $S^n \subset \R^{n+1}$; we will identify $T^\ast S^n$ with $T S^n$ throughout, using the round metric restricted from the  standard inner product on $\bR^{n+1}$. A point $\xi \in T^\ast S^n$ is given by $(u,v) \in \R^{n+1} \times \R^{n+1}$ such that $|u|=1$ and $u \cdot v = 0$. Choose a  smooth function $\theta : \R \to \R$ such that $\theta (r) = 0$ for $r < 1$ and $\theta (r) = 2\pi$ for $r > R$, some fixed constant $R>1$. A representative for $[\tau^2] \in \pi_0 (\Diff_c (T^\ast S^n))$ is given by:
$$
(u,v) \longmapsto \begin{cases}
R^{\theta(|v|)}_{(u,v)}  (u,v) & \quad \text{if } v \neq 0 \\
(u,v) &\quad \text{if } v = 0
\end{cases} 
$$
where we apply the rotation to each vector seperately. 
As long as $v \neq 0$, the three vectors $u,v, u\times v$ are linearly independent. This means that for every $t \in [0,1]$, there is a well-defined map $S_t: T^\ast S^n \to T^\ast S^n $ given by 
$$
S_t (u,v) =
\begin{cases}
(R^{\theta(|v|)}_{((1-t)u+tu \times v,v)}  (u,v) & \quad \text{if } v \neq 0 \\
(u,v) &\quad \text{if } v = 0
\end{cases} 
$$
By construction, $S_0$ is our representative for $\tau^2$; on the other hand, for $t=1$, as $u$ is orthogonal to $u \times v$ and to $v$, $R^\theta_{(u\times v, v)}$ fixes $u$. Thus $S_1$ fixes each cotangent fibre set-wise; on any given fibre, it has support in an annulus, where it acts by a full twist. This can now by undone by deforming the function $\theta$ to be the constant $2 \pi$, giving a compactly supported isotopy to the identity.
\end{proof}

\section{The method of Kauffman and Krylov}

\subsection{Upper bounds: the smooth mapping class group}\label{sec:MCG}

Kauffman and Krylov \cite[Theorem 3]{KK} determine the group $\pi_0(\mathrm{Diff}_c(T^* S^n))$ up to extensions, for $n \geq 3$. For $n$ even their result
takes the form of an extension
$$0 \lra \pi_0(\mathrm{VDiff}_c(T^* S^n))  \lra \pi_0(\mathrm{Diff}_c(T^* S^n))  \overset{\mathrm{Var}}\lra \bZ^\times \lra 0$$
defining the group $\pi_0(\mathrm{VDiff}_c(T^* S^n))$, and a further extension
$$0 \lra \Theta_{2n+1} \lra  \pi_0(\mathrm{VDiff}_c(T^* S^n))   \overset{\chi_r}\lra \begin{cases}
0 & n=6\\
\bZ/2 \oplus \bZ/2 & n \equiv 0 \mod 8\\
\bZ/2 & \text{otherwise}
\end{cases}$$
describing this group. 
(We have used Cerf's theorem \cite{Cerf} that pseudoisotopy implies isotopy for simply-connected manifolds of dimension $\geq 5$ to state their result in terms of isotopy. To address a concern of the referee: while  Cerf's results are formulated only for closed manifolds, the methods extend to manifolds with boundary cf.\ \cite[p.~11]{HatcherWagoner}.) Here the homomorphism $\mathrm{Var}$ is defined using the action on homology described in Lemma \ref{lem:CotangentAction}, $\chi_r$ is a certain homomorphism which we will not need to analyse directly, and $\Theta_{2n+1}$ denotes the group of homotopy $(2n+1)$-spheres, which is isomorphic to $\pi_0(\mathrm{Diff}_\partial(D^{2n}))$ and corresponds to those compactly-supported diffeomorphisms of $T^*S^n$ which are supported in a disc. Of note to us will be the Kervaire sphere $[\Sigma_{\text{Kervaire}}] \in \Theta_{2n+1}$, defined for each even $n$: this is trivial if $n= 6,14,30$, and perhaps $62$, and has order 2 otherwise \cite[Theorem 1.3]{HHR}.

\subsection{Lower bounds: open book decompositions}

Let us write $M = D^*(S^n)$, with $\partial M = S^*(S^n)$. If 
$$f : (M, \partial M) \lra (M, \partial M)$$
is a map which is the identity on the boundary, let
$$X_f := \frac{M \times [0,1]}{(x, 1) \sim (f(x),0)} \cup_{\partial M \times S^1} (\partial M \times D^2)$$
be the manifold obtained from the mapping torus of $f$ by filling in the middle. If $f$ is a diffeomorphism then $X_f$ is a smooth manifold, and if $f'$ is smoothly isotopic to $f$ then $X_{f'}$ is diffeomorphic to $X_f$. 
Similarly in the topological category. If $f$ is a homotopy equivalence then $X_f$ need not be a manifold but will be a Poincar{\'e} duality complex, and depends up to homotopy equivalence only on the homotopy class of $f$. For the identity map we have 
$$X_{\mathrm{Id}_M} = (D^*(S^n) \times S^1) \cup_{S^*(S^n) \times S^1} (S^*(S^n) \times D^2) \cong S^{n+1} \times S^n,$$
so if $X_f$ is not smoothly, or topologically, or homotopically, equivalent to $S^{n+1} \times S^n$, then $[f]$ is not trivial in the corresponding mapping class group. This is the strategy proposed by Kauffman and Krylov \cite[\S 3]{KK} to study the orders of elements. In the case of powers of the Dehn twist the geometric input is as follows.

\begin{proposition}\label{prop:Openbook} Let $\#_{A_l} D^\ast(S^{n+1})$ denote the plumbing of $l$ copies of $D^\ast(S^{n+1})$ in an $A_l$ chain, i.e.~the Milnor fibre of the $A_l$ singularity. 
For $k \geq 2$ there are diffeomorphisms
$$X_{\tau^k} \cong \partial (\#_{A_{k-1}} D^\ast(S^{n+1})).$$
\end{proposition}

\begin{proof}
This will be familiar to symplectic topologists.
The plumbing of  $(k-1)$ copies of $D^\ast(S^{n+1})$ is diffeomorphic to the Milnor fibre of the $A_{k-1}$ singularity in $n+2$ variables, i.e.~$B_R(0) \cap \{ z_0^{k} + z_1^2 + \ldots + z_{n+1}^2 =1 \} \subset \bC^{n+2}$, for any suitably large $R$.
By projecting to the first coordinate, this is the total space of a Lefschetz fibration with fibre the $A_1$ Milnor fibre one dimension down, i.e.~$D^\ast(S^n)$, and $k$ Morse critical points, all with vanishing cycle the zero section $S^n \subset D^\ast(S^n)$. In particular, the total monodromy of this Lefschetz fibration is $\tau^k_{S^n}$. This equips the boundary of its total space, i.e.~$\partial (\#_{A_{k-1}} D^\ast(S^{n+1}))$, with the claimed open book decomposition.
\end{proof}

 For a general introduction to open book decompositions of Brieskorn spheres, see  \cite[Section 4]{Kwon-vanKoert}.

\begin{corollary}\label{cor:Orders}\mbox{}
\begin{enumerate}[(i)]
\item If $n \neq 2,6$ then $[\tau^2] \in \pi_0(\mathrm{hAut}_c(T^* S^n))$ is nontrivial.

\item We have $[\tau^4] = [\Sigma_{\text{Kervaire}}] \in \Theta_{2n+1} \leq  \pi_0(\mathrm{Diff}_c(T^* S^n))$.
\end{enumerate}
\end{corollary}
\begin{proof}
For (i), if $\tau^2 : (M, \partial M) \to (M, \partial M)$ were homotopic relative to the boundary to the identity, then we would have a homotopy equivalence
$$S^*(S^{n+1}) \cong X_{\tau^2} \simeq X_{\mathrm{Id}_M} \cong S^{n+1} \times S^n,$$
so we must show that this is not the case. If it were the case then there would be a map
$$S^*(S^{n+1}) \simeq S^{n+1} \times S^n \overset{\pi_2}\lra S^n$$
which is an isomorphism on $n$th homology, and this would give a fibre homotopy trivialisation of the spherical fibration $S^*(S^{n+1}) \to S^{n+1}$. By a theorem of Milnor and Spanier \cite[Theorem 2]{MS} this implies that $n+1$ is ($1$,) $3$  or $7$.

For (ii), first note that by the group extensions described in Section \ref{sec:MCG} we have $[\tau^4] \in \Theta_{2n+1} \leq  \pi_0(\mathrm{Diff}_c(T^* S^n))$, so $[\tau^4] = [\Sigma]$ for some homotopy sphere $\Sigma$. In other words $\tau^4$ may be represented up to isotopy by a diffeomorphism $f$ of a disc $D^{2n} \subset D^*(S^n)$ relative to its boundary, and $\Sigma$ is the exotic sphere obtained by gluing two $(2n+1)$-discs along the diffeomorphism
$$f \cup \mathrm{Id}_{D^{2n}} : S^{2n} = D^{2n} \cup_\partial D^{2n} \lra D^{2n} \cup_\partial D^{2n} = S^{2n}.$$
But then $X_{\tau^4} \cong X_f$, and $X_f \cong X_{\mathrm{Id}_M} \# \Sigma = (S^{n+1} \times S^n) \# \Sigma$. Now recall that
$$\Sigma_{\text{Kervaire}} = \partial (\#_{A_2} D^*(S^{n+1})).$$
On the other hand, 
$$X_{\tau^4} \cong \partial (\#_{A_3} D^*(S^{n+1})).$$
`Cutting off' the plumbing of the first two spheres from $\#_{A_3} D^*(S^{n+1})$ leaves a copy of $D^\ast B^{n+1}$. This is glued to $\#_{A_2} D^*(S^{n+1})$ along  its vertical boundary, $\partial_v (D^\ast(B^{n+1}))  \cong  B^{n+1} \times S^{n}$, which naturally lies in $\partial (\#_{A_2} D^*(S^{n+1}))$. Let us write $\partial_h (D^\ast(B^{n+1}))  \cong  S^n \times B^{n+1}$ for the horizontal boundary, and $\partial_c \cong S^n \times S^n$ for the corners. We get
\begin{multline*}
X_{\tau^4} 
\cong \partial_h (D^*(B^{n+1})) \cup_{\partial_c} \big( \partial (\#_{A_2} D^*(S^{n+1})) \backslash \partial_v (D^\ast(B^{n+1}))  \big) \\
\cong (S^n \times B^{n+1}) \cup_{S^n \times S^n} \big( \Sigma_{\text{Kervaire}} \backslash (B^{n+1} \times S^{n}) \big)
\cong (S^{n} \times S^{n+1}) \# \Sigma_{\text{Kervaire}}
\end{multline*}

Thus we have a diffeomorphism $(S^{n+1} \times S^n) \# \Sigma_{\text{Kervaire}} \cong (S^{n+1} \times S^n) \# \Sigma$, but then by e.g.\ \cite[Theorem A]{Schultz} it follows that $[\Sigma]=[\Sigma_{\text{Kervaire}}]$.
\end{proof}

\begin{proof}[Proof of Theorem \ref{thm:main}]
The case $n$ odd was covered in Section \ref{seq:CotangentAction}, so we suppose that $n$ is even. As we have mentioned (i) is due to Seidel \cite[Theorem 4.17]{Seidel_graded}, and we have discussed the $n=2$ and $6$ cases in Section \ref{sec:nIs2}. So we suppose that $n \geq 4$ is even.

For (ii), we start by using $[\tau^4] = [\Sigma_{\text{Kervaire}}]$ from Corollary \ref{cor:Orders} (ii), and that $[\Sigma_{\text{Kervaire}}] \in \Theta_{2n+1}$ has order 2, to say that $[\tau]$ has order dividing 8. If the Kervaire sphere is nontrivial (i.e.\ $n \neq 6, 14, 30$ and perhaps $62$) it follows $[\tau]$ has order precisely 8. Otherwise $[\tau]$ has order dividing 4. By Corollary \ref{cor:Orders} (i) if $n \neq 6$ then $[\tau]$ does not have order 2 (even up to homotopy), so has order precisely 4. In the remaining case $n=6$ the extensions of Section \ref{sec:MCG} show that $[\tau^2] \in \Theta_{13} \cong \bZ/3$, so as $[\tau^2]^2 = [\tau^4] = [\Sigma_{\text{Kervaire}}]$ is trivial in this group of odd order, $[\tau^2]$ is trivial too.

As $\tau^2$ is smoothly isotopic to the identity for $n=2$ or $6$, it is also topologically isotopic and homotopic to the identity in these cases, which establishes part (a) of (iii) and (iv). As $\tau^4$ is smoothly isotopic to a diffeomorphism supported in a disc, it follows from the Alexander trick that $[\tau^4]$ is trivial in $\pi_0(\mathrm{Homeo}_c(T^* S^n))$ and hence in $\pi_0(\mathrm{hAut}_c(T^* S^n))$. As $[\tau^2]$ is nontrivial in $\pi_0(\mathrm{hAut}_c(T^* S^n))$ by Corollary \ref{cor:Orders} (i), part (b) of (iii) and (iv) follows.
\end{proof}

\section{Almost-complex diffeomorphisms}\label{sec:ACDiffeo}

There is a further kind of mapping class group between the symplectic and smooth ones, making use of the standard complex structure $\ell$ on $T^*S^n$, namely: the group of isotopy classes of pairs of a compactly-supported  diffeomorphism $f : T^* S^n \to T^*S^n$ and a compactly-supported path of complex structures from $f^*\ell$ to $\ell$. We call this the \emph{almost-complex mapping class group} of $T^* S^n$. Equivalently, but somewhat more formally, it may be described as the fundamental group based at $\ell$ of the homotopy orbit space
$$AC_c(T^*S^n) \hcoker \mathrm{Diff}_c(T^* S^n)$$
given by the action of the group of compactly-supported diffeomorphisms on the space of almost-complex structures on the manifold $T^*S^n$ which agree with $\ell$ outside a compact set. (We consider an almost-complex structure as an endomorphism of the tangent bundle squaring to $-\mathrm{Id}$, and give the space of such the compact-open topology.)

 In this case our results are not as conclusive as Theorem \ref{thm:main}, but informally say that in the almost-complex mapping class group the Dehn twist $\tau$ has order approximately $n!$. As the results are not completely conclusive we do not try to carefully distinguish the various cases of $n$ modulo 8: doing so would lead to some slight improvements.
 
\begin{question}
What is the precise order of the Dehn twist in the almost-complex mapping class group?
\end{question}

\subsection{Upper bounds}

To obtain an upper bound we consider the space $AC_c(T^*S^n)$ of almost-complex structures on $T^*S^n$ which agree with the standard one outside a compact set. 

\begin{lemma}
If $n$ is even, $\pi_1(AC_c(T^*S^n))$ is finite of order dividing $2 \cdot n!$.
\end{lemma}
\begin{proof}
We may instead work with $AC_\partial(D^*(S^n))$, the space of almost-complex structures on $D^*(S^n)$ which are standard near the boundary. The space of almost-complex structures on the vector space $\bR^{2n}$ is $\tfrac{O(2n)}{U(n)}$. As the tangent bundle of $D^*(S^n)$ is trivial, there is an equivalence $AC_\partial(D^*(S^n)) \simeq \mathrm{map}_\partial(D^*(S^n), \tfrac{O(2n)}{U(n)})$, to the space of continuous maps extending a fixed map $\ell: S^*(S^n) \to \tfrac{O(2n)}{U(n)}$. As $D^*(S^n)$ is obtained from $S^*(S^n)$ by attaching an $n$-cell and a $2n$-cell, there is a homotopy fibre sequence
$$\Omega^{2n}\big(\tfrac{O(2n)}{U(n)}\big) \lra \mathrm{map}_\partial \big(D^*S^n, \tfrac{O(2n)}{U(n)}\big) \lra \Omega^n\big(\tfrac{O(2n)}{U(n)}\big),$$
and hence an exact sequence
$$\cdots \lra \pi_{2n+1}\big(\tfrac{O(2n)}{U(n)}\big) \lra \pi_1(AC_\partial(D^*(S^n))) \lra \pi_{n+1}\big(\tfrac{O(2n)}{U(n)}\big) \lra \cdots.$$

We have $\pi_{n+1}\big(\tfrac{O(2n)}{U(n)}\big) \cong \pi_{n+1}\big(\tfrac{O}{U}\big)$ and the latter are given by Bott periodicity \cite{Bott1} as $\bZ/2$ if $n \equiv 6 \mod 8$ and 0 otherwise.
(Recall that $O/U \simeq \Omega^2(\mathbb{Z} \times BO)$ and that          we are assuming $n$ is even when we
     say ``otherwise''.)
 From the table in \cite{Harris} we find that $\pi_{2n+1}\big(\tfrac{O(2n)}{U(n)}\big)$ has order $2 \cdot n!$ if $n \equiv 0 \mod 4$ and order $n!$ if $n \equiv 2 \mod 4$. Combining these gives the conclusion.
\end{proof}

\begin{corollary}
In the almost-complex mapping class group of $T^*S^n$ with $n$ even the Dehn twist has finite order dividing $2^4 \cdot n!$.
\end{corollary}
\begin{proof}
The long exact sequence on homotopy groups for the homotopy orbit space $AC_c(T^*S^n) \hcoker \mathrm{Diff}_c(T^* S^n)$ has a portion
$$ \pi_1(AC_c(T^*S^n)) \lra \pi_1(AC_c(T^*S^n) \hcoker \mathrm{Diff}_c(T^* S^n)) \lra \pi_0(\mathrm{Diff}_c(T^* S^n)).$$
The diffeomorphism $\tau^8$ is smoothly isotopic to the identity by Theorem \ref{thm:main}, so by this exact sequence its class as an almost-complex diffeomorphism  comes from $\pi_1(AC_c(T^*S^n))$, which by the lemma above is finite of order dividing $2 \cdot n!$. The conclusion follows.
\end{proof}

\subsection{Lower bounds}

Our discussion of open books can be upgraded to the almost-complex setting, as follows. If $f$ is an almost-complex diffeomorphism, the space $X_f$ comes naturally equipped with an almost-contact structure. In the case of $\tau$, this is compatible with Proposition \ref{prop:Openbook}: the almost-contact structure is the one induced by the standard contact form on $S_R(0) \cap \{ z_0^{k} + z_1^2 + \ldots + z_{n+1}^2 =1 \} \subset \bC^{n+2}$. (This is also discussed in \cite{Kwon-vanKoert}.) Let $[\xi_k]$ be the almost-contact structure on $X_{\tau^k}$. If $k \equiv \pm 1\mod 8$ then $X_{\tau^k} \cong S^{2n+1}$ and from \cite[Section 2]{Ustilovsky}, based on calculations in \cite{Morita}, we have the following.

\begin{lemma}\label{lem:Ustilovsky}
 If $k \equiv \pm 1 \mod 8$ then the almost-contact structure $[\xi_k]$ is given by $$(k-1)/2 \in \pi_{2n+1} ( SO(2n+1) / U(n)) \cong \bZ/d,$$ where
$$
d =
\begin{cases}
n! & \quad \text{if } n \equiv 0 \mod 4 \\
n! / 2 & \quad \text{if } n \equiv 2 \mod 4.
\end{cases}
$$
\end{lemma}

\begin{corollary}
In the almost-complex mapping class group of $T^*S^n$ with $n$ even the Dehn twist has order divisible by $n!/2^2$.
\end{corollary}
\begin{proof}
Let us write the order of $\tau$ in the almost-complex mapping class group as $2^r \cdot s$ with $s$ odd. Then $\tau$ is equivalent to $\tau^{1+k 2^r s}$ in this group and so $X_\tau \cong X_{\tau^{1+ k 2^r s}}$ as almost-contact manifolds for any $k \in \bZ$. If $r \geq 3$ then by Lemma \ref{lem:Ustilovsky} it follows that $2^{r-1} s$ is divisible by $n!/2$, so $2^r s$ is divisible by $n!$. If $r < 3$ then we can take $k=2^{3-r}$ so that $k 2^r s = 8 s$ and hence by Lemma \ref{lem:Ustilovsky} it again follows that $8 s$ is divisible by $n!$, so $2^r s$ is divisible by $n!/2^{3-r}$. The result follows as $r \geq 1$  by  Theorem \ref{thm:main}.
\end{proof}

\bibliographystyle{amsalpha}
\bibliography{biblio}

\end{document}